\documentclass[amsfonts,reqno,11pt]{amsart}

\usepackage{epsfig}
\usepackage{amssymb}
\usepackage{xcolor}

\newtheorem{thm}{Theorem}[section]

\newtheorem{lem}[thm]{Lemma}

\newtheorem{rem}[thm]{Remark}

\newtheorem{prob}[thm]{Problem}

\newcommand{\R}{{\Bbb R}}

\begin{document}

 \title[On bodies with congruent sections or projections]{On bodies with congruent sections or projections}

 \author{Ning Zhang}

 \address{Ning Zhang, Mathematical Sciences Research Institute, Berkeley, CA 94720, USA}
 \email{nzhang2@ualberta.ca}

  \subjclass[2010]{52A20, 52A38}

\keywords{convex body, congruent sections, unique determination}

\begin{abstract} 
 In this paper, we construct two convex bodies $K$ and $L$ in $\R^n$, $n\geq 3$, such that their projections $K|H$, $L|H$ onto every subspace $H$ are congruent, but nevertheless, $K$ and $L$ do not coincide up to a translation or a reflection in the origin. This gives a negative answer to an old conjecture posed by Nakajima and S\"uss.
\end{abstract}

\maketitle

\section{Introduction and main results}
In this paper we give a negative answer to an old conjecture first considered by Nakajima \cite{Na} and S\"uss \cite{Su} in 1932 (see, for example, the book of R.~J.~Gardner ``Geometric tomography'' \cite[Problem 3.2, p.~125 and Problem 7.3, p.~289]{Ga}).

\begin{prob}\label{p1.1.12}
	Suppose that $2\leq k\leq n-1$ and that $K$ and $L$ are star bodies
	in $\R^n$ such that the section $K\cap H$ is congruent to $L\cap H$ for all $H\in G(n,k)$.
	Do $K$ and $L$ coincide up to a reflection in the origin?
\end{prob}
Here, $G(n,k)$ denotes the Grassmanian of $k$-dimensional subspaces of $\R^n$ and $K\cap H$ being congruent to $L\cap H$ means that there exists an orthogonal transformation $\varphi$ in $H$ such that $\varphi(K\cap H)$ is a translate of $L\cap H$. 

The following question is a ``dual version'' of Problem \ref{p1.1.12}.

\begin{prob}\label{p1.1.13}
	Suppose that $2\leq k\leq n-1$ and that $K$ and $L$ are convex bodies
	in $\R^n$ such that the orthogonal projection $K| H$ is congruent to $L| H$ for all $H\in G(n,k)$.
	Do $K$ and $L$ coincide up to a translation and a reflection in the origin?
\end{prob}

Myroshnychenko and Ryabogin in \cite{MR} gave a positive answer to both problems in the class of polytopes. For Problem \ref{p1.1.12}, Gardner \cite[Theorem 7.1.1]{Ga} obtained that $K,L$ coincide up to a translation in the case when $K,L$ are convex bodies and $K\cap H$ is a translate of $L\cap H$ for every $H$. Nakajima \cite{Na} and S\"uss \cite{Su} showed that the answer to Problem \ref{p1.1.13} is affirmative in the case when $K|H$ is a translate of $L|H$ for every $H$ (see also Ryabogin \cite{R1}). Later Golubyatnikov \cite{Go} gave an affirmative answer to Problem \ref{p1.1.13} in $\R^3$ when $K|H$ is a direct congruent of $L|H$ for every $H$ and $K|H$, $L|H$ do not have $SO(2)$-symmetries. Here, the set $K|H$ is directly congruent to $L|H$ means that there exists a special orthogonal transformation $\varphi\in SO(2)$ in $H$ such that $\varphi(K| H)=L| H$; and $A$ has $SO(2)$-symmetry means that there exists a $\phi\in SO(2)$ such that $\phi(A)$ is parallel to $A$. In the cases $k=3,4$, some partial results were also obtained by Alfonseca, Cordier, and Ryabogin in \cite{ACR} and \cite{ACR2}. See also \cite{R2}. In particular, using our construction one can also obtain negative answers to Problem 3, 4, 5, 9, 10 posed there.
 
Our main results are the following.

\begin{thm}\label{t1.3}
	There exist two convex bodies $K,L\in\R^n$ ($n\geq 3$) such that $K\neq L$ and $K\neq -L$ but nevertheless for every $\xi\in S^{n-1}$ there is an orthogonal transformation $\phi_\xi\in O(n-1)$ in $\xi^\perp$ satisfying $\phi_\xi(K\cap\xi^\perp)=L\cap\xi^\perp$.
\end{thm}
Here, $\xi^\perp$ is the hyperplane perpendicular to $\xi$.
\begin{thm}\label{t1.4}
	There exist two convex bodies $K,L\in \R^n$ ($n\geq 3$) such that $K\neq L$ and $K\neq -L$ but nevertheless for every $\xi\in S^{n-1}$ there is an orthogonal transformation $\phi_\xi\in O(n-1)$ in $\xi^\perp$ satisfying $\phi_\xi(K|\xi^\perp)=L|\xi^\perp$.
\end{thm}

If we consider rotations only (no translation), Ryabogin gave an affirmative answer to Problem \ref{p1.1.12} and \ref{p1.1.13} in $\R^3$ (see \cite{R} for details). However, we show here that in $\R^n$ ($n\geq 4$), a counterexample can be constructed.
\begin{thm}\label{t1.5}
	There exists two convex bodies $K,L\in \R^n$ ($n\geq 4$) such that $K\neq L$ and $K\neq -L$ but neverthless for every $\xi\in S^{n-1}$ there is a rotation $\phi_\xi\in SO(n-1)$ in $\xi^\perp$ satisfying $\phi_\xi(K\cap\xi^\perp)=L\cap\xi^\perp$.
\end{thm}

Our idea resembles the one from papers by Ryabogin and Yaskin \cite{RY} and Nazarov, Ryabogin, and Zvavitch \cite{NRZ}, but it is more elaborated. It is more difficult to verify that the required bodies have the corresponding pairwise congruent sections (or projections) rather than that the intrinsic volumes of sections (or projections) coincide.

The paper is organized as follows. In Section \ref{s2} we formulate some definitions and prove our main auxiliary results. Section \ref{s3} gives the proof of Theorem \ref{t1.3}. We prove Theorem \ref{t1.5} in Section \ref{s4}.

\section{Notation and auxiliary lemmata}\label{s2}
We will set the following standard notation. The unit sphere in the $n$-dimensional Euclidean space $\R^n$, $n\geq 3$, is $S^{n-1}$. We denote $\{e_1,e_2,\dots,e_n\}$ to be the standard basis in $\R^n$. For any $\xi\in S^{n-1}$, we denote $\xi^\perp:=\{x\in\R^n:\langle x, \xi\rangle=0\}$ to be a hyperplane which is perpendicular to $\xi$. Here, $\langle x, \xi\rangle$ is the usual inner product in $\R^n$. We also denote $\xi^\perp_+:=\{x\in\R^n:\langle x, \xi\rangle\geq 0\}$ to be the ``upper half'' of $\xi^\perp$ and $\xi^\perp_-:=\{x\in\R^n:\langle x, \xi\rangle\leq 0\}$ to be the ``lower half'' of $\xi^\perp$. The Grassmann manifold of $k$-dimensional subspaces in $\R^n$ is denoted by $G(n,k)$. $O(k)$ and $SO(k)$, $2\leq k\leq n$, are the subgroups of the orthogonal group $O(n)$ and the sphecial orthogonal group $SO(n)$ in $\R^n$.

Given a $\xi\in S^{n-1}$, for any $\eta\in S^{n-1}\cap \xi^\perp$, in $\xi^\perp$ the reflection in $\eta$, denoted by $\tilde{\phi}_{\xi,\eta}$, is an orthogonal transformation such that $\tilde{\phi}_{\xi,\eta}(x)=-x+2\langle x, \eta \rangle \eta, \forall x\in \xi^\perp$. A set $A\subset \xi^\perp$ is reflection symmetric in a direction $\eta$ if $\tilde{\phi}_{\xi,\eta}(A)=A$. For $H\in G(n-1,2)$ being a $2$-dimensional subspace of $\xi^\perp$ ($n\geq 4$), in $\xi^\perp$ the reflection in $H$, denoted by $\tilde{\phi}_{\xi,H}$, is an orthogonal transformation such that for any $x\in \xi^\perp$, $x+\tilde{\phi}_{\xi,H}(x)\in H$.

We refer to \cite[Chapter 1]{Ga} for the deﬁnitions of convex and star bodies. A {\em convex body} is a compact set $K$ with non-empty interior such that the segment joining every pair of points in $K$ is contained in $K$. 

A set $K \subset \R^n$ is said to be {\em star-shaped} with respect to a point $o$ if the line segment from $o$ to any point in $K$ is contained in $K$. A {\em shar body} $K \subset \R^n$ is a nonempty, compact, star-shaped set with respect to the origin, the {\em radial function} of a star body $K\subset \R^n$ at $x\in \R^n$ is defined as $\rho_K(x):=\sup\{s\geq 0:sx\in K\}$.

Here and below, let $E$ be an ellipsoid in $\R^n$ under the standard basis $\{e_1,e_2,\dots,e_n\}$, that is, $E:=\{x\in \R^n: \sum_{i=1}^n\frac{x_i^2}{a_i^2}=1\}$
with $a_1>a_2>\cdots>a_n>0$. We denote $\rho_E$ to be the radial function of $E$.

We will denote $C(S^{n-1})$ to be the class of all continuous functions on $S^{n-1}$ and $C^k(S^{n-1})$, $k\geq 1$, to be the class of all functions with $k$ continuous derivatives on $S^{n-1}$. The {\em level set} of a function $f\in C(S^{n-1})$ at a value $\tau$ is $\mathcal{L}_\tau f:=\{\theta\in S^{n-1}: f(\theta)=\tau\}$.

The {\em radial extension} of a set $E\subset S^{n-1}$ is denoted by $rad(E):=\{\lambda x: \lambda\geq 0,~x \in E\}\subset \R^n$.

Given $\xi\in S^{n-1}$, an {\em elliptic cone} around $\xi$ is a cone $C\subset \R^n$ such that $C\cap (\xi^\perp+t\xi)$ is an $(n-1)$-dimensional ellipsoid for every $t>0$.

\begin{lem}\label{l2.3}
	Let $E$ be an ellipsoid in $\R^n$. Then if $\tau_1\in (a_2,a_1)$ the radial extension of the upper level set $\{\rho_E(\theta)\geq \tau_1\}\cap (e_1)^\perp_+$ is an elliptic cone around $e_1$ and if $\tau_2\in (a_n,a_{n-1})$ the radial extension of the lower level set $\{\rho_E(\theta)\leq \tau_2\}\cap (e_n)^\perp_+$ is an elliptic cone around $e_n$.
\end{lem}
\begin{proof}
	When $\tau_1\in (a_2,a_1)$, the radial extension of $\{\rho_E(\theta)\geq \tau_1\}\cap (e_1)^\perp_+$ is, for $\lambda\geq 0$,
	\begin{align*}
	&\lambda(\{x\in (e_1)^\perp_+: \sum_{i=1}^n x_i^2\geq \tau_1^2\}\cap \{x\in (e_1)^\perp_+: \sum_{i=1}^n\frac{x_i^2}{a_i^2}=1\})\\
	&=\{x\in (e_1)^\perp_+: \sum_{i=1}^n x_i^2 \geq \tau_1^2 (\sum_{i=1}^n\frac{x_i^2}{a_i^2}) \}=\{x\in (e_1)^\perp_+: \sum_{i=1}^n(1-\frac{\tau_1^2}{a_i^2})x_i^2\geq 0\}\\
	&=\{x\in (e_1)^\perp_+: (1-\frac{\tau^2}{a_1^2})x^2\geq \sum_{i=2}^n(\frac{\tau_1^2}{a_i^2}-1)x_i^2\},
	\end{align*}
	which is an elliptic cone around $e_1$ in $(e_1)^\perp_+$.
	
	When $\tau_2\in (a_n,a_{n-1})$, the radial extension of $\{\rho_E(\theta)\leq \tau_2\}\cap (e_n)^\perp_+$ is, for $\lambda\geq 0$,
	\begin{align*}
	&\lambda(\{x\in (e_n)^\perp_+: \sum_{i=1}^n x_i^2\geq \tau_2^2\}\cap \{x\in (e_n)^\perp_+: \sum_{i=1}^n\frac{x_i^2}{a_i^2}=1\})\\
	&=\{x\in (e_n)^\perp_+: \sum_{i=1}^n x_i^2 \leq \tau_2^2 (\sum_{i=1}^n\frac{x_i^2}{a_i^2})\}\\
	&=\{x\in (e_n)^\perp_+: \sum_{i=1}^n(1-\frac{\tau_2^2}{a_i^2})x_i^2\leq 0\}\\
	&=\{x\in (e_n)^\perp_+: (\frac{\tau_2^2}{a_n^2}-1)x_n^2\geq \sum_{i=1}^{n-1}(1-\frac{\tau_2^2}{a_i^2})x_i^2 \},
	\end{align*}
	which is an elliptic cone around $e_n$ in $(e_n)^\perp_+$.
\end{proof}

\begin{lem}\label{l2.5}
	Let $E$ be an ellipsoid in $\R^n$. For any subsphere $S^2\cap \xi^\perp$ intersecting $\{\rho_E(\theta)<a_{n-1}\}$ and $\{\rho_E(\theta)>a_2\}$, set $\tau_2:=\min\{\rho_E(\theta): \theta\in S^{n-1}\cap\xi^\perp\}$ and $ \tau_1:=\max\{\rho_E(\theta): \theta\in S^{n-1}\cap\xi^\perp\}$. Then we have
	$$
	L_{\tau_1}\rho_E\cap S^2\cap \xi^\perp=\{\eta_1,-\eta_1\}
	$$
	and
	$$
	L_{\tau_2}\rho_E\cap S^2\cap \xi^\perp=\{\eta_2,-\eta_2\}.
	$$
		\begin{figure}[h]
			\center	\includegraphics[width=0.8\linewidth]{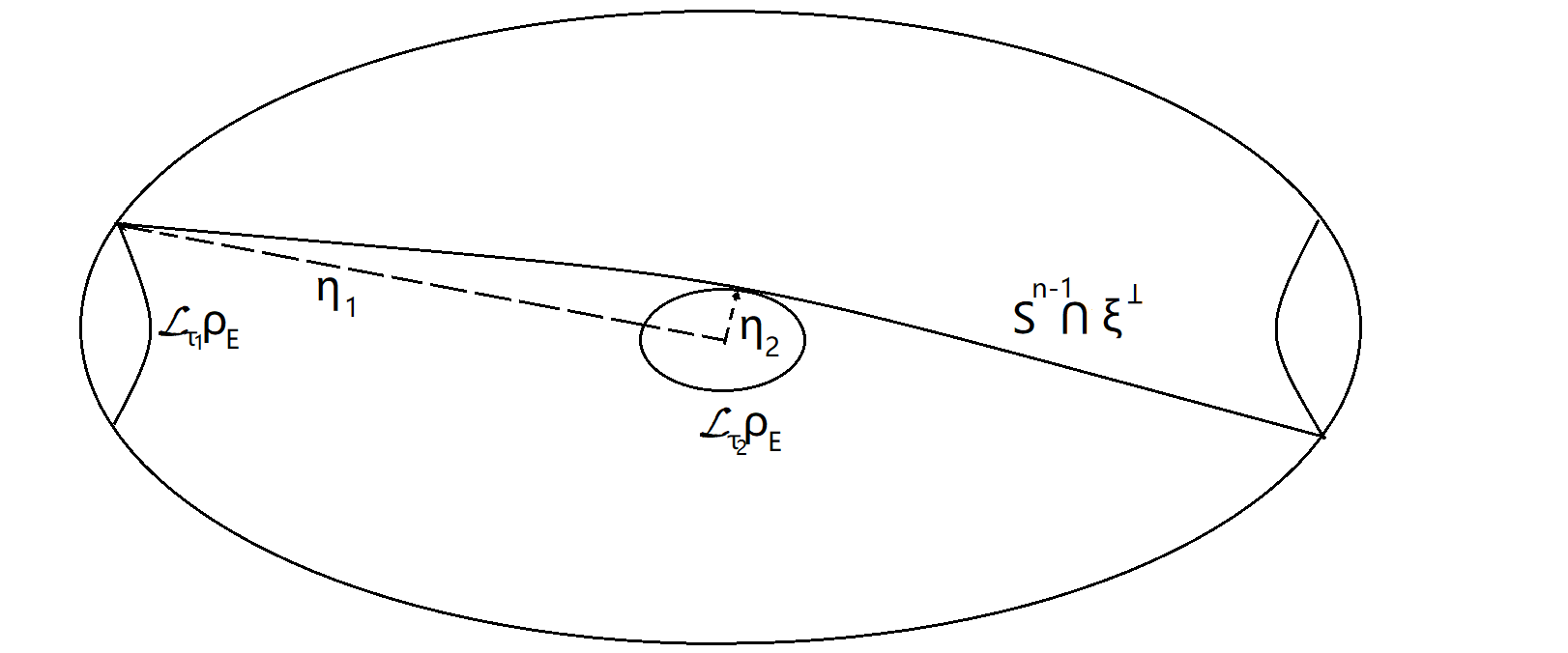}
			\caption{}\label{f1.1}
		\end{figure}
		
	Moreover, for any $\tau\in [\tau_2,\tau_1]$, we have
	$$
	\tilde{\phi}_{\xi,\eta_1}(L_{\tau}\rho_E\cap\xi^\perp)=L_{\tau}\rho_E\cap\xi^\perp=\tilde{\phi}_{\xi,\eta_2}(L_{\tau}\rho_E\cap\xi^\perp).
	$$
\end{lem}
\begin{proof}
	Since the subsphere $S^{n-1}\cap \xi^\perp$ intersects $\{\rho_E(\theta)<a_{n-1}\}$ and $\{\rho_E(\theta)>a_2\}$, we have
	$$
	a_n\leq\tau_2<a_{n-2}<a_2<\tau_1\leq a_1.
    $$
    
    First, if $\tau_1=a_1$, 
    $$
    L_{\tau_1}\rho_E\cap S^2\cap \xi^\perp=\{e_1,-e_1\}
    $$ 
    and if $\tau_2=a_n$,
	$$
	L_{\tau_2}\rho_E\cap S^2\cap \xi^\perp=\{e_n,-e_n\}.
	$$

	Now by Lemma \ref{l2.3}, if $\tau_2\in (a_n,a_{n-1})$, the radial extension of $\{\rho_E\leq \tau_2\}\cap (e_n)^\perp_+$ is an elliptic cone around $e_n$ in $(e_n)^\perp_+$. Hence, there exists a unique point $\eta_2\in S^{n-1}$ such that $S^{n-1}\cap \xi^\perp\cap \{\rho_E\leq \tau_2\}\cap (e_n)^\perp_+=\{\eta_2\}$. By symmetry of the ellipsoid $E$, $S^{n-1}\cap \xi^\perp\cap \{\rho_E\leq \tau_2\}\cap (e_n)^\perp_-=\{-\eta_2\}$. Similarly, if $\tau_1\in (a_2,a_1)$, there exists a unique point $\eta_1\in S^{n-1}$ such that $S^{n-1}\cap \xi^\perp\cap \{\rho_E\geq \tau_1\}=\{\eta_1,-\eta_1\}$.
	
	Therefore, the level sets $L_{\tau_1}\rho_E\subset \{\rho_E(\theta)>a_2\}$ and $L_{\tau_2}\rho_E\subset \{\rho_E(\theta)<a_{n-1}\}$ satisfy
	$
	L_{\tau_1}\rho_E\cap S^{n-1}\cap \xi^\perp=\{\eta_1,-\eta_1\}
	$
	and
	$
	L_{\tau_2}\rho_E\cap S^{n-1}\cap \xi^\perp=\{\eta_2,-\eta_2\}
	$ (see Figure \ref{f1.1}).

	Note that $E\cap \xi^\perp$ is an $(n-1)$-dimensional ellipsoid with maximum radius at directions $\{\eta_1,-\eta_1\}$ and minimum radius  at directions $\{\eta_2,-\eta_2\}$; hence, $E\cap \xi^\perp$ is reflection symmetric in directions $\eta_1$ and $\eta_2$: for any $\theta\in S^{n-1}\cap \xi^\perp$,
	$$
	\rho_E(\tilde{\phi}_{\xi,\eta_1}(\theta))=\rho_E(\theta)
	$$
	and
	$$
	\rho_E(\tilde{\phi}_{\xi,\eta_2}(\theta))=\rho_E(\theta).
	$$
	This implies (see Figure 2)
	$$
	\tilde{\phi}_{\xi,\eta_1}(L_{\tau}\rho_E\cap \xi^\perp)=L_{\tau}\rho_E=\tilde{\phi}_{\xi,\eta_2}(L_{\tau}\rho_E\cap \xi^\perp).
	$$
	\begin{figure}[h]
		\center	\includegraphics[width=0.8\linewidth]{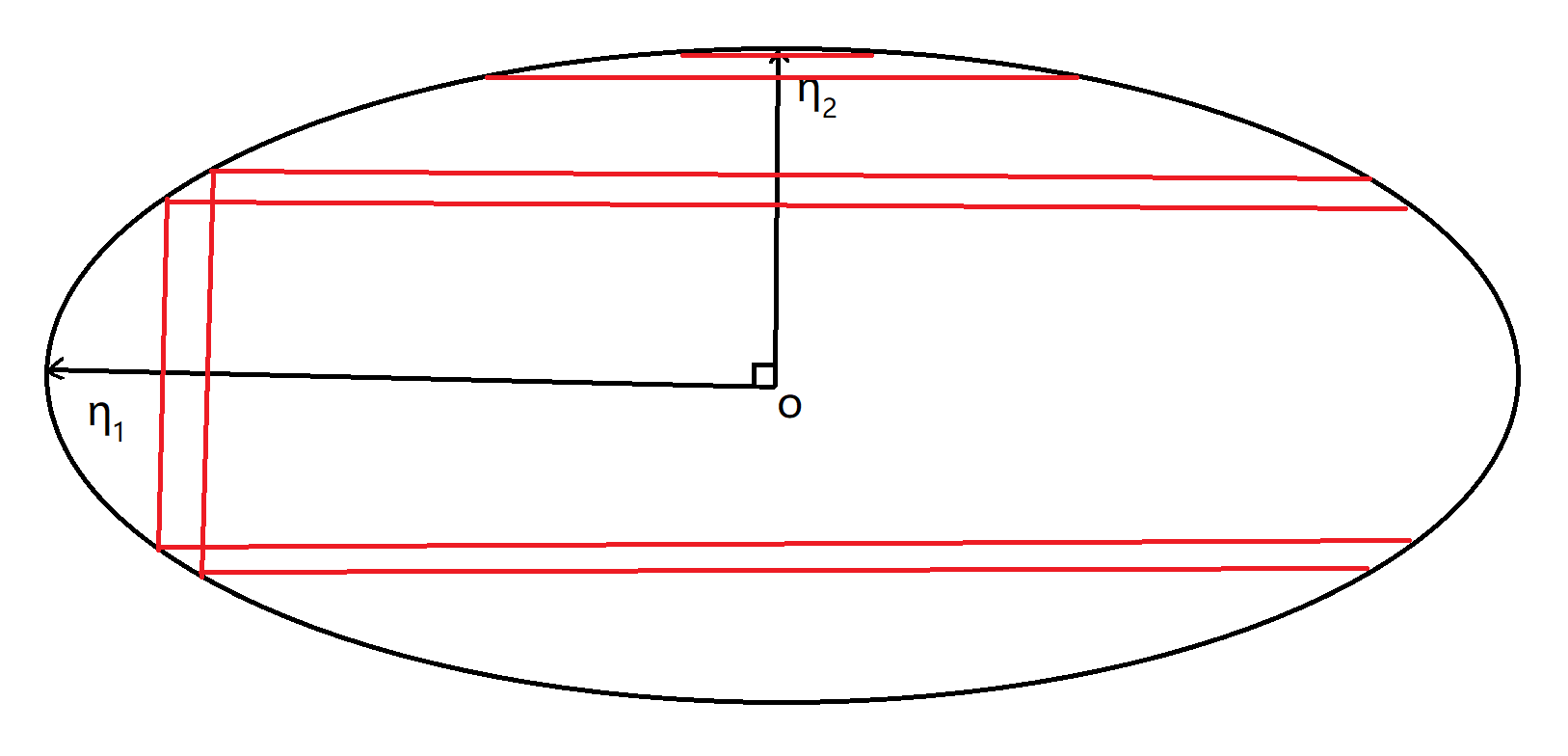}
		\caption{}\label{f1.2}
	\end{figure}	
\end{proof}

\section{Proof of Theorem \ref{t1.3}}\label{s3}
Our bodies will be some special perturbations of ellipsoids. For a fixed ellipsoid in $\R^n$ with base $\{e_1,e_2,\dots,e_n\}$ 
$$
E=\{x\in \R^n: \sum_{i=1}^n\frac{x_i^2}{a_i^2}=1\},
$$
where $a_1>a_2>\cdots>a_n>0$. Fix $\delta>0$ small enough such that 
$$
\delta<\min_{2\leq i\leq n}\{\frac{a_i-a_{i-1}}{2}\}.
$$
We define two sets on $S^{n-1}$ as 
$$
I_1:=\{\theta\in S^{n-1}: a_1-\rho_E(\theta)<\delta\mbox{ and }\langle\theta,e_1\rangle>0\}
$$
$$
I_2:= \{\theta\in S^{n-1}: \rho_E(\theta)-a_n<\delta\mbox{ and }\langle\theta,e_n\rangle>0\}.
$$
We define the convex body $K\subset  \R^n$ ($n\geq 3$) (see Figure 3) via its radial function $\rho_K\in C^2_+(S^{n-1})$ as
\begin{align*}
 \rho_K(\theta)=\rho_E(\theta)-\varepsilon(a_1-\rho_E(\theta)-\delta)^3 &\quad \mbox{when } \theta\in I_1;\\
 \rho_K(\theta)=\rho_E(\theta)+\varepsilon(\rho_E(\theta)-a_n-\delta)^3 &\quad \mbox{when } \theta\in I_2;\\
 \rho_K(\theta)=\rho_E(\theta) & \quad \mbox{otherwise}.
\end{align*}

Next, we define another convex body $L\subset  \R^n$ ($n\geq 3$) (see Figure 3) whose radial function $\rho_L\in C^2_+(S^{n-1})$ is
\begin{align*}
\rho_L(\theta)=\rho_E(\theta)-\varepsilon(a_1-\rho_E(\theta)-\delta)^3 &\quad \mbox{when } \theta\in -I_1;\\
	\rho_L(\theta)=\rho_E(\theta)+\varepsilon(\rho_E(\theta)-a_n-\delta)^3 &\quad \mbox{when }\theta\in I_2;\\
	\rho_L(\theta)=\rho_E(\theta) & \quad \mbox{otherwise}.
\end{align*}
Here, we choose $\varepsilon$ sufficiently small to guarantee that $K$ and $L$ are convex (see \cite[Page 41]{KY}).

\begin{figure}[h]
	\center	\includegraphics[width=0.4\linewidth]{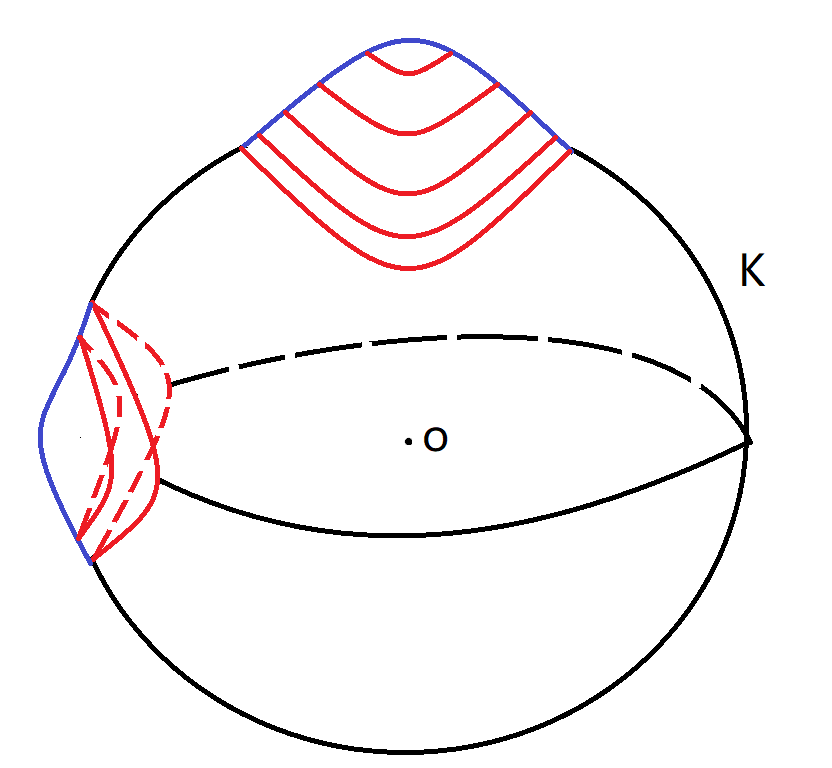}
	\includegraphics[width=0.4\linewidth]{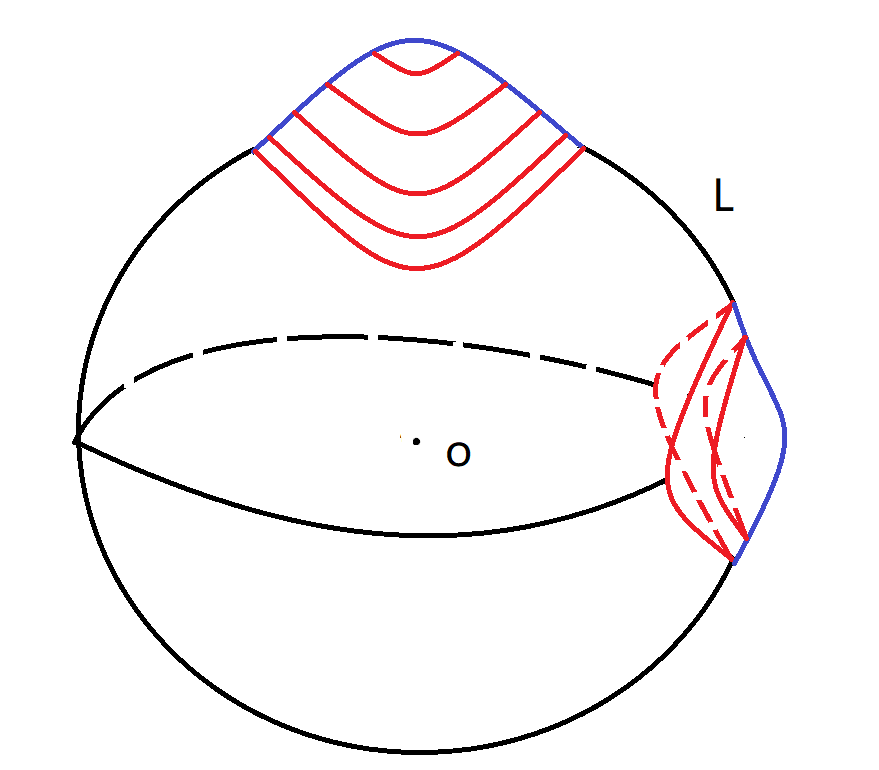}
	\caption{}\label{f1.3}
\end{figure}

Note that when $x\in I_1$, we have $-x\in -I_1$, then
\begin{align*}
\rho_K(-x)&=\rho_E(-x)-\varepsilon(a_1-\rho_E(-x)-\delta)^3\\
&=\rho_E(x)-\varepsilon(a_1-\rho_E(x)-\delta)^3=\rho_L(x)
\end{align*}
and
\begin{align*}
	\rho_K(x)-\rho_L(x)&=\rho_E(x)-(\rho_E(x)-\varepsilon(a_1-\rho_E(x)-\delta)^3)\\
	&=\varepsilon(a_1-\rho_E(x)-\delta)^3>0.
\end{align*}
When $x\in I_2$, we have
$$
\rho_L(x)=\rho_E(x)+\varepsilon(\rho_E(x)-a_n-\delta)^3=\rho_K(x)
$$
and $-x\notin I_2\cup I_1$, hence
\begin{align*}
	\rho_L(-x)=\rho_E(-x)=\rho_E(x)\neq \rho_E(x)+\varepsilon(\rho_E(x)-a_n-\delta)^3=\rho_K(x).
\end{align*}
Therefore, there exists some $\theta_1\in I_1$ and $\theta_2\in I_2$ such that $\rho_K(\theta_1)\neq \rho_L(\theta_1)$ and $\rho_K(\theta_2)\neq \rho_L(-\theta_2)$, which means $K$ is not equal to $L$ up to a reflection in the origin.

Now define the function
$$
h_1(x)=x-\varepsilon(a_1-x-\delta)^3,
$$
which is a strictly increasing function when $x\in [a_1-\delta,a_1]$ since
$$ 
h'_1(x)=1+3\varepsilon(a_1-x-\delta)^2>0
$$
and the function
$$
h_2(x)=x+\varepsilon(x-a_n-\delta)^3,
$$
which is a strictly increasing function when $x\in [a_n,a_n+\delta]$ since
$$ 
h'_2(x)=1+3\varepsilon(x-a_n-\delta)^2>0.
$$

Observe that for $\tau\in (a_1-\delta,a_1]$, the level set of $\rho_K(\theta)$ at value $h_1(\tau)$ is the level set of $\rho_E(\theta)$ at value $\tau$, i.e. 
\begin{equation}\label{e1}
\mathcal{L}_{h_1(\tau)} \rho_K=\mathcal{L}_{\tau} \rho_E\cap I_1.
\end{equation}
Similarly, for $\tau\in (a_1-\delta,a_1]$ we have 
\begin{equation}\label{e2}
\mathcal{L}_{h_1(\tau)} \rho_L=\mathcal{L}_{\tau} \rho_E\cap (-I_1)
\end{equation} 
 and for $\tau\in [a_n,a_n+\delta)$
\begin{equation}\label{e3}
\mathcal{L}_{h_2(\tau)} \rho_K=\mathcal{L}_{h_2(\tau)} \rho_L=\mathcal{L}_{\tau} \rho_E\cap I_2.
\end{equation} 

We consider several cases.

{\bf Case 1.} If the subsphere $S^{n-1}\cap\xi^\perp$ does not intersect $I_1$, we set $\phi_\xi$ to be the identity map. Then for any $\theta\in S^{n-1}\cap \xi^\perp$, $\rho_K(\theta)=\rho_L(\theta)=\rho_E(\theta)$.

{\bf Case 2.} If the subsphere $S^{n-1}\cap\xi^\perp$ only intersects $I_1$ but not $I_2$, note that for any $\theta\in S^{n-1}\cap \xi^\perp\cap (-I_1)$, $\rho_K(-\theta)= h_2(\rho_E(-\theta))= h_2(\rho_E(\theta))=\rho_L(\theta)$ and for any $\theta\in S^{n-1}\cap \xi^\perp\backslash (-I_1)$, $\rho_K(-\theta)=\rho_E(\theta)=\rho_L(\theta)$. Then we set $\phi_\xi$ to be the reflection in the origin.

{\bf Case 3.} If the subsphere $S^{n-1}\cap\xi^\perp$ intersects both $I_1$ and $I_2$, we need to choose a proper reflection. Now we set $\tau_1=\max\{\rho_E(\theta):\theta\in S^{n-1}\cap \xi^\perp\}$ and $\tau_2=\min\{\rho_E(\theta):\theta\in S^{n-1}\cap \xi^\perp\}$. By Lemma \ref{l2.5}, there exists $\eta_1=S^{n-1}\cap \xi^\perp\cap\mathcal{L}_{\tau_1} \rho_E\cap I_1$ and $\eta_2=S^{n-1}\cap \xi^\perp\cap\mathcal{L}_{\tau_2} \rho_E\cap I_2$ such that
$$
\tilde{\phi}_{\xi,\eta_1}(L_{\tau}\rho_E\cap \xi^\perp)=L_{\tau}\rho_E=\tilde{\phi}_{\xi,\eta_2}(L_{\tau}\rho_E\cap \xi^\perp)
$$
for any $\tau\in [\tau_2,\tau_1]$. Thus, together with Equation (\ref{e1}) and Equation (\ref{e2}), taking $\phi_\xi=\tilde{\phi}_{\xi,\eta_2}$, we have for any $\tau\in (a_1-\delta,\tau_1]$
\begin{align*}
	&\tilde{\phi}_{\xi,\eta_2}(\mathcal{L}_{ h_1(\tau)}\rho_K\cap \xi^\perp)\\
	=& \tilde{\phi}_{\xi,\eta_2}(\mathcal{L}_{\tau}\rho_E\cap \xi^\perp\cap I_1)\\
	=& \mathcal{L}_{\tau}\rho_E\cap \xi^\perp\cap (-I_1)\quad \mbox{(see Figure \ref{f1.2})}\\
	=&\mathcal{L}_{ h_1(\tau)}\rho_L\cap \xi^\perp.
\end{align*}
Hence, $\rho_K(\tilde{\phi}_{\xi,\eta_2}(\theta))=\rho_L(\theta)$ for $\theta\in -I_1\cap\xi^\perp$. On the other hand, by Lemma \ref{l2.5} and Equation (\ref{e3}), taking $\phi_\xi=\tilde{\phi}_{\xi,\eta_2}$, we have for any $\tau\in [\tau_2,a_n+\delta)$
\begin{align*}
	&\tilde{\phi}_{\xi,\eta_2}(\mathcal{L}_{ h_2(\tau)}\rho_K\cap \xi^\perp)\\
	=&\tilde{\phi}_{\xi,\eta_2}(\mathcal{L}_{\tau}\rho_E\cap \xi^\perp\cap I_2)\\
	=& \mathcal{L}_{\tau}\rho_E\cap \xi^\perp\cap I_2\quad \mbox{(see Figure \ref{f1.2})}\\
	=&\mathcal{L}_{h_2(\tau)}\rho_L\cap \xi^\perp.
\end{align*}
Hence, $\rho_K(\tilde{\phi}_{\xi,\eta_2}(\theta))=\rho_L(\theta)$ for $\theta\in I_2\cap\xi^\perp$. Finally, for any $\theta\in S^2\cap \xi^\perp \backslash (-I_1\cup I_2)$, $\tilde{\phi}_{\xi,\eta_2}(\theta)\in S^2\cap \xi^\perp \backslash (I_1\cup I_2)$, then $\rho_K(\tilde{\phi}_{\xi,\eta_2}(\theta))=\rho_E(\theta)=\rho_L(\theta)$. Therefore, taking $\phi_\xi=\tilde{\phi}_{\xi,\eta_2}$, we have
$$
\tilde{\phi}_{\xi,\eta_2}(K\cap \xi)=L\cap \xi.
$$

\begin{rem}
	The choices of $h_1$ and $h_2$ are not unique.
\end{rem}

\section{Proof of Theorem \ref{t1.5}}\label{s4}
The rest of this paper will show that the bodies constructed in Section \ref{s3} are exactly the required bodies in Theorem \ref{t1.5}.

Recall that given $\xi\in S^{n-1}$, in $\xi^\perp$, the reflection map in a unit vector $\eta$ is $\tilde{\phi}_{\xi,\eta}$ and the reflection map in a $2$-dimensional plane $H$ is $\tilde{\phi}_{\xi,H}$ (see Section \ref{s2}).
\begin{lem}\label{l2.6}
Let $n$ be even. For any $\xi\in S^{n-1}$ and $\eta\in S^{n-1}\cap \xi^\perp$, $$\tilde{\phi}_{\xi,\eta}\in SO(n-1).$$ Let $n$ be odd. For any $\xi\in S^{n-1}$ and $H\in G(n-1,2)$ in $\xi^\perp$, $$\tilde{\phi}_{\xi,H}\in SO(n-1).$$
\end{lem}
\begin{proof}
	Let $n$ is even. For any $\eta\in S^{n-1}\cap \xi^\perp$ the reflection matrix $M_{\tilde{\phi}_{\xi,\eta}}$ of $\tilde{\phi}_{\xi,\eta}$  under the orthonormal basis $\{\zeta_1,\cdots,\zeta_{n-2},\eta\}$ of the subspace $\xi^\perp$ is
	$$
	\begin{bmatrix}
	-1 & 0 & \dots  & 0 &  0 \\
	0 & -1 & \dots  &0 & 0 \\
	\vdots & \vdots & \ddots & \vdots  & \vdots \\
	0 & 0 & \dots & -1  & 0\\
	0 &0 & \dots& 0   & 1
	\end{bmatrix}_{(n-1)\times (n-1)}.
	$$
	Then $\det(M_{\tilde{\phi}_{\xi,\eta}})=(-1)^{n-2}=1$ implies $\tilde{\phi}_{\xi,\eta}\in SO(n-1)$.
	
	Let $n$ be odd. For any $H\in G(n-1,2)$ being the $2$-dimensional subspace of $\xi^\perp$, set $\eta_2,\eta_3\in S^{n-1}\cap H$ to be perpendicular verctors. Then the reflection matrix $M_{\tilde{\phi}_{\xi,H}}$ of $\tilde{\phi}_{\xi,H}$  under the orthonormal basis $\{\zeta_1,\cdots,\zeta_{n-3},\eta_2,\eta_3\}$ of the subspace $\xi^\perp$ is
	$$
	\begin{bmatrix}
	-1 & 0 & \dots  & 0 &  0 \\
	\vdots & \ddots  & \vdots & \vdots  & \vdots \\
	0  & \dots &-1 & 0  & 0\\
	0 & \dots & 0 & 1  & 0\\
	0 & \dots &0 & 0   & 1
	\end{bmatrix}_{(n-1)\times (n-1)}.
	$$
	Hence $\det(M_{\tilde{\phi}_{\xi,H}})=(-1)^{n-3}=1$ implies $\tilde{\phi}_{\xi,H}\in SO(n-1)$.
\end{proof}

\begin{proof}[Proof of Theorem \ref{t1.5}]
Our goal is to find a rotation for every subsphere $S^{n-1}\cap \xi^\perp$. We again discuss in serveral cases.

{\bf Case 1.} If the subsphere $S^{n-1}\cap\xi^\perp$ does not intersect $I_1$, for any $\theta\in S^{n-1}\cap \xi^\perp$, we set $\phi_\xi$ to be the identity map, so $\rho_K(\theta)=\rho_E(\theta)=\rho_L(\theta)$.

{\bf Case 2.} If the subsphere $S^{n-1}\cap\xi^\perp$ only intersects $I_1$ but not $I_2$, we now choose a rotation. Now we set $\tau_1=\max\{\rho_E(\theta):\theta\in S^{n-1}\cap \xi^\perp\}$, then by Lemma \ref{l2.5}, $E\cap \xi^\perp$ is reflection symmetric in the direction $\eta_1=S^{n-1}\cap \xi^\perp\cap\mathcal{L}_{\tau_1} \rho_E\cap I_1$. Moreover since $S^{n-1}\cap\xi^\perp\cap E$ is an $(n-1)$-dimensional ellipsoid, denoted by $E_\xi$, we can construct an orthonormal basis $\mathcal{B}_\xi=\{\eta_1,\zeta_2,\cdots, \zeta_{n-1},\xi\}$ such that $E_\xi$ can be written as
$$
E_\xi=\{y\in \xi^\perp: \frac{y^2_1}{\tau^2_1}+\sum_{i=2}^{n-1}\frac{y^2_i}{\tilde{a}^2_i}=1\},
$$
where $(y_1,\dots,y_{n-1},y_n)$ is the coordinate vector of $y\in\R^n$ relative to basis $\mathcal{B}_\xi$. Here, $\tilde{a}_i=\rho_E(\zeta_i)$ for $i=2,\dots,n-1$. Note that $E_\xi$ is reflection symmetric in directions $\eta_1,\zeta_2,\cdots, \zeta_{n-1}$.

If $n$ is even, taking $\zeta_{n-1}$ from basis $\mathcal{B}_\xi$, by Lemma \ref{l2.6}, the reflection in $\zeta_{n-1}$ is exactly a rotation in $\xi^\perp$. Therefore, by the same argument in Lemma \ref{l2.5}, we have $\tilde{\phi}_{\xi,\zeta_{n-1}}(E_\xi)=E_\xi$. Then for any $\theta\in S^{n-1}\cap \xi^\perp\cap (-I_1)$, $\tilde{\phi}_{\xi,\zeta_{n-1}}(\theta)\in S^{n-1}\cap \xi^\perp\cap I_1$, that is  $$\rho_K(\tilde{\phi}_{\xi,\zeta_{n-1}}(\theta))= h_1(\rho_E(\tilde{\phi}_{\xi,\zeta_{n-1}}(\theta)))=  h_1(\rho_E(\theta))=\rho_L(\theta);$$ and for any $\theta\in S^{n-1}\cap \xi^\perp\backslash (-I_1)$, $\rho_K(\tilde{\phi}_{\xi,\zeta_{n-1}}(\theta))=\rho_E(\theta)=\rho_L(\theta)$. Hence taking $\phi_\xi=\tilde{\phi}_{\xi,\zeta_{n-1}}\in SO(n-1)$, we have
$$
\tilde{\phi}_{\xi,\zeta_{n-1}}(K\cap \xi^\perp)=L\cap \xi^\perp.
$$

If $n$ is odd, taking $\zeta_{n-2},\zeta_{n-1}$ from basis $\mathcal{B}_\xi$, by Lemma \ref{l2.6}, the reflection in the $2$-dimensional subspace $H$ spaned by $\zeta_{n-2},\zeta_{n-1}$ is exactly a rotation on $S^{n-1}\cap\xi^\perp$. Therefore, by the same argument in Lemma \ref{l2.5}, we have $\tilde{\phi}_{\xi,H}(E_\xi)=E_\xi$. Then for any $\theta\in S^{n-1}\cap \xi^\perp\cap (-I_1)$, $\tilde{\phi}_{\xi,H}(\theta)\in S^{n-1}\cap \xi^\perp\cap I_1$, that is  $$\rho_K(\tilde{\phi}_{\xi,H}(\theta))= h_1(\rho_E(\tilde{\phi}_{\xi,H}(\theta)))= h_1(\rho_E(\theta))=\rho_L(\theta);$$ and for any $\theta\in S^{n-1}\cap \xi^\perp\backslash (-I_1)$, $\rho_K(\tilde{\phi}_{\xi,H}(\theta))=\rho_E(\theta)=\rho_L(\theta)$. Hence, taking $\phi_\xi=\tilde{\phi}_{\xi,H}\in SO(n-1)$, we have
$$
\tilde{\phi}_{\xi,H}(K\cap \xi^\perp)=L\cap \xi^\perp.
$$

{\bf Case 3.} If the subsphere $S^{n-1}\cap\xi^\perp$ intersects both $I_1$ and $I_2$, we again choose a rotation. Here we use the same construction in Section \ref{s3}. Set $\tau_1=\max\{\rho_E(\theta):\theta\in S^{n-1}\cap \xi^\perp\}$ and $\tau_2=\min\{\rho_E(\theta):\theta\in S^{n-1}\cap \xi^\perp\}$. By Lemma \ref{l2.5}, there exists $\eta_1=S^{n-1}\cap \xi^\perp\cap\mathcal{L}_{\tau_1} \rho_E\cap I_1$ and $\eta_2=S^{n-1}\cap \xi^\perp\cap\mathcal{L}_{\tau_2} \rho_E\cap I_2$ such that $E\cap \xi^\perp$ is reflection symmetric in directions $\eta_1,\eta_2$. 
If $n$ is even, by Lemma \ref{l2.6} taking $\phi_\xi=\tilde{\phi}_{\xi,\eta_2}\in SO(n-1)$, we have
$$
\tilde{\phi}_{\xi,\eta_2}(K\cap \xi^\perp)=L\cap \xi^\perp.
$$

If $n$ is odd, since $E\cap\xi^\perp$ is an $(n-1)$-dimensional ellipsoid, denoted by $E_\xi$, we can construct an orthonormal basis $\mathcal{B}_\xi=\{\eta_1,\zeta_2,\cdots, \zeta_{n-2},\eta_2,\xi\}$ such that $E_\xi$ can be written as
$$
E_\xi=\{y\in \xi^\perp: \frac{y^2_1}{\tau^2_1}+\sum_{i=2}^{n-2}\frac{y^2_i}{\tilde{a}^2_i}+\frac{y^2_{n-1}}{\tau^2_2}=1\}
$$
where $(y_1,\dots,y_{n-2},y_{n-1},y_n)$ is the coordinate vector of $y\in\R^n$ relative to basis $\mathcal{B}_\xi$. Here, $\tilde{a}_i=\rho_E(\zeta_i)$ for $i=2,\dots,n-2$. Note that $E_\xi$ is reflection symmetric in directions $\eta_1,\zeta_2,\cdots, \zeta_{n-2},\eta_2$.

Finally, taking $\zeta_{n-2},\eta_2$ from basis $\mathcal{B}_\xi$, by Lemma \ref{l2.6}, the reflection in the $2$-dimensional subspace $H$ spaned by $\zeta_{n-2},\eta_2$ is exactly a rotation on $S^{n-1}\cap\xi^\perp$. Therefore, by the argument similar to Lemma \ref{l2.5}, we have $\tilde{\phi}_{\xi,H}(E_\xi)=E_\xi$. Then for any $\theta\in S^{n-1}\cap \xi^\perp\cap (-I_1)$, $\tilde{\phi}_{\xi,H}(\theta)\in S^{n-1}\cap \xi^\perp\cap I_1$, that is
$$
\rho_K(\tilde{\phi}_{\xi,H}(\theta))= h_1(\rho_E(\tilde{\phi}_{\xi,H}(\theta)))= h_1(\rho_E(\theta))=\rho_L(\theta);
$$ 
for any $\theta\in S^{n-1}\cap \xi^\perp \cap I_2$, $\tilde{\phi}_{\xi,H}(\theta)\in S^{n-1}\cap \xi^\perp\cap I_2$, that is
$$
\rho_K(\tilde{\phi}_{\xi,H}(\theta))= h_2(\rho_E(\tilde{\phi}_{\xi,H}(\theta)))= h_2(\rho_E(\theta))=\rho_L(\theta);
$$
and for any $\theta\in S^{n-1}\cap \xi^\perp\backslash (-I_1\cup I_2)$, $\rho_K(\tilde{\phi}_{\xi,H}(\theta))=\rho_E(\theta)=\rho_L(\theta)$. Hence taking $\phi_\xi=\tilde{\phi}_{\xi,H}\in SO(n-1)$, we have
$$
\tilde{\phi}_{\xi,H}(K\cap \xi^\perp)=L\cap \xi^\perp.
$$
\end{proof}

\begin{rem}
	In our construction all the reflections or rotations are involutions, that is $\phi\in O(n-1)$ with $\phi^2$ being the identity transformation. It is natural to ask the following question: Let $K,L\subset \R^n$ be convex bodies and $k\in [2,n-1]$ be integer numbers. If for any $H\in G(n,k)$, $K\cap H$ and $L\cap H$ are congruent, does there exist an integer $l\in [0,n-1]$ such that $K$ and $L$ coincide up to reflection in some $l$-dimensional subsphere?
\end{rem}
\section{Acknowledgement}

I would like to express my gratitude to my MSRI-adviser, Prof.~Dmitry Ryabogin, for fruitful discussions; especially, the question ``why the ellipsoid is special?'' gives me the initial idea of this paper.

This material is based upon work supported by the National Science Foundation under Grant No. DMS-1440140 while the author is in residence at the Mathematical Sciences Research Institute in Berkeley, California, during the Fall 2017 semester.

\end{document}